\documentclass[11pt]{article}
\usepackage[a4paper]{anysize}\marginsize{3.5cm}{3.5cm}{1.3cm}{2.5cm}
\usepackage[latin1]{inputenc}

\pdfpagewidth=\paperwidth \pdfpageheight=\paperheight
\usepackage{amssymb}

\makeatletter
\renewcommand\@seccntformat[1]{\csname the#1\endcsname.\enspace}
\renewcommand\@begintheorem[2]{\trivlist\item[\hskip\labelsep{\bfseries#1 #2.}]\it}
\renewcommand\@opargbegintheorem[3]{\trivlist\item[\hskip\labelsep{\bfseries#1 #2}] {\bfseries(#3).}\enspace\it\ignorespaces}
\makeatother
\renewenvironment{abstract}{\begin{quote}\hrulefill\par\footnotesize\textbf{\abstractname.}}{\par\vskip-0.5\baselineskip\hrulefill\end{quote}}

\newtheorem{introtheorem}{Theorem}

\newtheorem{thm}{Theorem}[section]

\newtheorem{lemma}[thm]{Lemma}
\newtheorem{proposition}[thm]{Proposition}

\newcommand\mkthm[2]{\newenvironment{#1}{\begin{#2}\rm}{\end{#2}}}
 \mkthm{definition}{tdefinition}
         \mkthm{remark}{tremark}
       \mkthm{example}{texample}
     \mkthm{question}{tquestion}

\newenvironment{proof}[1][Proof]{\trivlist\item[\hskip\labelsep{\textit{#1.}}]}{\hspace*{\fill}$\Box$\endtrivlist}

\renewcommand\ge{\geqslant}  % schräge Variante
\renewcommand\le{\leqslant}  % schräge Variante
\renewcommand\bar{\overline}

\newcommand\grant[1]{{\renewcommand\thefootnote{}\footnotetext{#1.}}}
\newcommand\keywords[1]{{\renewcommand\thefootnote{}\footnotetext{\textit{Keywords:} #1.}}}
\newcommand\subclass[1]{{\renewcommand\thefootnote{}\footnotetext{\textit{Mathematics Subject Classification (2010):} #1.}}}

\newcommand\C{\mathbb C}
\newcommand\Q{\mathbb Q}
\newcommand\R{\mathbb R}
\newcommand\Z{\mathbb Z}

\newcommand\be[1][@{\;}r@{\;}c@{\;}l@{\;}l@{\;}]{$$\everymath{\displaystyle}\renewcommand\arraystretch{1.2}\begin{array}{#1}}
\newcommand\ee{\end{array}$$}

\newcommand\compact{\itemsep=0cm \parskip=0cm}
\newcommand\tfrac[2]{{\textstyle\frac{#1}{#2}}}
\newcommand\set[1]{\left\{#1\right\}}

\newcommand\matr[1]{\left(\begin{array}{*{20}{c}} #1 \end{array}\right)}

\newcommand\into{\hookrightarrow}

\newcommand\unitmatrix{{\mathchoice{\rm 1\mskip-4mu l}{\rm 1\mskip-4mu l}{\rm 1\mskip-4.5mu l}{\rm 1\mskip-5mu l}}}
\newcommand\tmod[1]{\;({\rm mod}~#1)}

\newcommand\tline{\noalign{\vskip0.4ex}\hline\noalign{\vskip0.65ex}}
\newcommand\abs[1]{\left|#1\right|}
\newcommand\Bigabs[1]{\Big|#1\Big|}
\newcommand\liste[3]{\mbox{$#1_{#2},\dots,#1_{#3}$}}

\newenvironment{bycases}{\left\{\begin{array}{@{}l@{\quad}l}}{\end{array}\right.}

\newcommand\eps{\varepsilon}
\newcommand\smallmatr[2]{\bigl({#1 \atop #2}\bigr)}

\newcommand\newop[2]{\newcommand#1{\mathop{\rm #2}\nolimits}}

\newop\End{End}
\newop\Fix{Fix}
\newop\GL{GL}
\newop\SL{SL}
\newop\fix{\#\Fix}  % Anzahl der Fixpunkte
\newop\tr{tr}
\newop\id{id}

\newcommand\fnum{\widetilde f}  % die komplexe Zahl, die dem Endomorphismus f entspricht
\newcommand\hilbertsymbol[3]{{#1,\, #2 \overwithdelims() #3}}

\begin{document}

   \title{Fixed points of endomorphisms \\
      on two-dimensional complex tori}
   \author{\normalsize Thomas Bauer, Thorsten Herrig}
   \date{\normalsize \today}
   \maketitle
   \thispagestyle{empty}
   \grant{The second author was supported by Studienstiftung des deutschen Volkes}
   \subclass{14A10, 14K22, 14J50}

\begin{abstract}
   In this paper we investigate fixed-point numbers of
   endomorphisms on
   complex tori.
   Specifically, motivated by the
   asymptotic perspective that has turned out in recent years to be so
   fruitful in
   Algebraic Geometry,
   we study how the number of fixed points
   behaves when the endomorphism is iterated.
   Our first result shows that
   the fixed-points function of an endomorphism on a
   two-dimensional complex torus can have only three different
   kinds of
   behaviours, and we characterize these behaviours in terms of
   the analytic eigenvalues.
   Our second result focuses on simple abelian surfaces and
   provides criteria for the fixed-points behaviour
   in terms of the possible types of endomorphism algebras.
\end{abstract}

%*****************************************************************************

\section*{Introduction}

   Given a holomorphic map $f:X\to X$ on a complex variety $X$,
   one of the natural questions about $f$ is how many
   fixed points it has.
   This number may, as expected, vary a lot between
   different endomorphisms, but
   it is a recurring theme in Algebraic Geometry that one
   hopes for much more regularity when
   adopting an asymptotic perspective.
   Examples for the fruitfulness of this approach are
   questions about base loci \cite{ELMMP:base-loci},
   growth of higher cohomology \cite{Fernex-Kuronya-Lazarsfeld:higher},
   syzygies \cite{Ein-Lazarsfeld:asymptotic-syzygies}
   and Betti numbers \cite{Ein-Erman-Lazarsfeld:Betti}.
   Concerning the question of fixed points, a natural asymptotic
   point of view consists in considering large iterates $f^n$
   of a given map. Specifically,
   denoting by $\fix(f)$ the number of fixed points of a map $f$,
   the question becomes:
   \begin{quote}\it
      What is the
      asymptotic
      behaviour of the fixed-points function
      \be
         n\mapsto \fix(f^n)
      \ee
      where $f^n=f\circ\dots\circ f$ denotes
      $n$-th iterate of
      $f$.
   \end{quote}
   The growth
   of the fixed-points function
   is also of interest in
   purely analytic contexts (e.g.~\cite{Shub-Sullivan}).
   In the present paper we consider it when
   $f$ is a holomorphic map on a complex torus.
   As is customary (cf.~\cite{BL:fixed})
   we set $\fix(f)=0$, if the fixed-points set
   is infinite, i.e., if $f$ fixes an analytic subvariety of positive
   dimension.

   Consider for instance the
   multiplication map $m_X:X\to X$, $x\mapsto mx$,
   on a complex torus $X$ of dimension $g$,
   for a given integer $m\ge 2$.
   Its fixed points are the $(m-1)$-torsion points, and hence
   the fixed-points number
   \be
      \fix((m_X)^n) = (m^n-1)^{2g}
   \ee
   grows exponentially with $n$.
   It is natural to wonder whether this is typical for
   endomorphisms on complex tori, and what other behaviour, if any, might
   occur.
   For two-dimensional complex tori we provide
   a complete answer:

\begin{introtheorem}\label{thm:tori}
   Let $X$ be a two-dimensional complex torus
   and let $f:X\to X$ be a non-zero endomorphism. Then
   the fixed-points function $n\mapsto\fix(f^n)$ has one of
   the following three behaviours:
   \begin{itemize}
   \item[\rm(B1)]
      It grows exponentially
      in $n$, i.e., there are real constants $A,B>1$
      and an integer $N$ such that
      for all $n\ge N$,
      \be
         A^n \le \fix(f^n) \le B^n
         \,.
      \ee
      In this case
      both
      eigenvalues of $f$
      (i.e., of its analytic representation
      $\rho_a(f)\in M_2(\C)$)
      are of absolute value $\ne 1$.
   \item[\rm(B2)]
      It is a periodic function.
      In this case
      the non-zero eigenvalues of $f$ are roots of unity,
      and they are contained in the set
      of $k$-th roots of unity where
      $k\in\set{1,\dots,6,8,10,12}$.
      %$n\in\set{1, 2, 3, 4, 5, 6, 8, 10, 12}$.
      %$k\le 12$, $k\notin\set{7,9,11}$.

   \item[\rm(B3)]
      It is of the form
      \be
         \fix(f^n)=
         \begin{bycases}
           0,    & \mbox{if } n\equiv 0\pmod r \\
           h(n), & \mbox{otherwise}
         \end{bycases}
      \ee
      where $r\ge 2$ is an integer
      and $h$ is an exponentially growing function.
      In this case one of the eigenvalues of $f$ is of absolute
      value $>1$ and the other is a root of unity.
   \end{itemize}
   All three behaviours occur already in the projective case,
   i.e., on abelian surfaces.
\end{introtheorem}

   The fact that the eigenvalues govern the behaviour of the
   fixed-points function is a
   consequence of the Holomorphic
   Lefschetz Fixed-Point Formula (see
   Prop.~\ref{prop:fix-and-eigenvalues}).
   The main point of the theorem is
   that the stated cases are the only ones that can occur.

   For simple abelian surfaces there are
   only three non-trivial types of possible endomorphism algebras,
   and it is desirable to know about the fixed-points behaviour in
   terms of these types. Our second result contains this
   information:

\begin{introtheorem}\label{thm:abelian}
   Let $X$ be a simple abelian surface.
   Then the fixed-points function of any non-zero endomorphism
   $f\in\End(X)$
   is either exponential {\rm(B1)} or periodic {\rm(B2)}, but
   has never behaviour {\rm(B3)}. Specifically, we have:
   \begin{itemize}
   \item[\rm(a)]
      Suppose that $X$ has real multiplication, i.e.,
      $\End_\Q(X)=\Q(\sqrt d)$ for a square-free integer $d>0$.
      Then $\fix(f^n)$ is periodic if
      $f=\pm\id_X$, and it grows exponentially otherwise.
   \item[\rm(b)]
      Suppose that $X$ has indefinite quaternion multiplication,
      i.e., $\End_\Q(X)$ is of the form $\Q+i\Q+j\Q+ij\Q$, where
      $i^2=\alpha\in\Q\setminus\set 0$, $j^2=\beta\in\Q\setminus\set 0$ with
      $ij=-ji$ and $\alpha>0$, $\alpha\ge\beta$. Write $f\in\End(X)$ as
      $f=a+bi+cj+dij$ with $a,b,c,d\in\Q$. Then $\fix(f^n)$ is
      periodic if
      $|a+\sqrt{b^2\alpha+c^2\beta-d^2\alpha\beta}|=1$, and
      it grows exponentially otherwise.
   \item[\rm(c)]
      Suppose that $X$ has complex multiplication, and let
      $\sigma:\End(X)\into\C$ be an embedding.
      Then $f$ has periodic fixed-point
      behaviour if $|\sigma(f)|=1$, and it has
      exponential fixed-points growth otherwise.
   \end{itemize}
\end{introtheorem}

   We give a more detailed description of the three types in
   Sect.~\ref{sect:abelian},
   and we provide a list of the finitely many
   eigenvalues that occur in endomorphisms
   with periodic fixed-points behaviour (see Prop.~\ref{prop:periodic-cases}).

   Fixed points of endomorphisms on complex tori and abelian
   varieties have been studied previously by Birkenhake and Lange
   \cite{BL:fixed} with a focus on
   the classification of
   fixed-point free
   automorphisms.
   The question of fixed point numbers for iterates of
   endomorphisms on abelian varieties was first addressed
   in the preprint \cite{Ringler}; however, Theorem 1.2 in
   \cite{Ringler} is unfortunately erroneous.

%*****************************************************************************

\section{Fixed points and eigenvalues}

   Let $X$ be a complex torus of dimension $g$ and let
   $f:X\to X$ be a holomorphic map.
   The map $f$ is
   a translate $f=h+a$
   of a group endomorphism $h\in\End(X)$ by some $a\in X$, and
   we have
   (see \cite{BL:fixed})
   \be
      \fix(f)=\fix(h)
      \,.
   \ee
   So, as far as fixed point numbers are concerned,
   it is enough to consider endomorphisms.
   The following
   proposition allows one to determine fixed point
   numbers from the eigenvalues of the analytic representation.

\begin{proposition}\label{prop:fix-and-eigenvalues}
   Let $f:X\to X$ be an endomorphism of a $g$-dimensional complex
   torus, and let $\liste\lambda1g$ be the eigenvalues of its
   analytic representation (counted with algebraic
   multiplicities).
   Then we have for every integer $n\ge 1$,
   \be
      \fix(f^n)=\Bigabs{\prod_{i=1}^g (1-\lambda_i^n)}^2
      \,.
   \ee
\end{proposition}

\begin{proof}
   Thanks to the Holomorphic
   Lefschetz Fixed-Point Formula \cite[13.1.2]{LB},
   the fixed point number
   can be computed from the analytic
   representation $\rho_a(f)\in M_g(\C)$,
   \be
      \fix(f^n)=\abs{\det(\unitmatrix_g-\rho_a(f)^n)}^2
      \,.
   \ee
   As the eigenvalues of $\rho_a(f)^n$ are $\liste{\lambda^n}1g$, we get
   the asserted formula.
\end{proof}

   The proposition shows that the
   fixed-points function
   $n\mapsto\fix(f^n)$
   is governed by the size of the
   eigenvalues.
   If, for instance, it were to happen that all eigenvalues of
   $f$ are
   of absolute value bigger than 1, then clearly $\fix(f^n)$ grows
   exponentially. (This is the case for the multiplication maps
   mentioned in the introduction.)
   The following examples show, however, that eigenvalues of absolute
   value equal to 1 as well as less than 1 occur, too.

\begin{example} (Eigenvalues of absolute value $1$).
\label{example:eigenvalues-one}
   Take an elliptic curve $E$ and consider the complex torus
   $X=E\times E$. The endomorphism
   \be
      f: X\to X, \quad (x, y)\mapsto (x-y, x)
   \ee
   has the eigenvalues $\frac{1+\sqrt{-3}}2$ and
   $\frac{1-\sqrt{-3}}2$. Using
   Prop.~\ref{prop:fix-and-eigenvalues} we find that
   the fixed-points function of $f$ is periodic:
   \be
      \fix(f^n)=
         \begin{bycases}
            0,  & \mbox{ if } n \equiv 0 \pmod 6 \\
            1,  & \mbox{ if } n \equiv 1 \mbox{ or } n \equiv 5 \pmod 6 \\
            9,  & \mbox{ if } n \equiv 2 \mbox{ or } n \equiv 4 \pmod 6 \\
            16, & \mbox{ if } n \equiv 3 \pmod 6\,. \\
         \end{bycases}
   \ee
\end{example}

\begin{example} (Eigenvalues of absolute value $<1$).
\label{example:eigenvalues-less-than-one}
   Examples of this kind have been constructed by
   McMullen \cite[Sect.~4]{McMullen:dynamics} in order to exhibit
   degree 6 Salem numbers (see also \cite{Reschke:Salem}).
   Specifically, McMullen shows that
   for every integer
   $a\ge 0$, the polynomial
   \be
      P(t)=t^4+at^2+t+1
   \ee
   occurs as the characteristic polynomial of
   the rational representation of
   an endomorphism on a
   two-dimensional complex torus.
   One checks that
   the zeros of $P(t)$
   appear in conjugate pairs
   $\alpha,\bar\alpha,\beta,\bar\beta$ with
   $\abs\alpha<\abs\beta$.
   As the constant term of $P$ equals $1$, it follows that
   $\abs\alpha<1$.
\end{example}

   In fact,
   a closer analysis of the preceding example shows that
   arbitrarily small eigenvalues occur:

\begin{proposition}
   For every $\eps>0$ there exists a two-dimensional
   complex torus with an
   endomorphism that has a non-zero eigenvalue of absolute value less
   than $\eps$.
\end{proposition}

\begin{proof}
   We draw on the complex tori constructed in
   Example~\ref{example:eigenvalues-less-than-one}.
   An application of
   Rouch\'e's theorem shows that
   for all sufficiently large values of $a$ there exists a root
   of
   the polynomial
   \be
      t^4+at^2+t+1
   \ee
   in the disk of radius $\eps$ around the origin.
   This proves the proposition.
\end{proof}

   The following example shows that eigenvalues of absolute value
   $<1$ occur also in the
   projective case:

\begin{example}
   Let $E$ be an elliptic curve, and consider
   the abelian surface
   $X=E\times E$. Every matrix
   $\smallmatr{a \ b}{c \ d}\in\SL_2(\Z)$
   defines an automorphism $f$ of $X$.
   Its analytic characteristic
   polynomial is $t^2-(a+d)t+1$, and from this one checks that
   $f$
   has an eigenvalue arbitrarily close to 0 when $a+d$ is
   sufficiently large.

   With a little more effort
   the same behaviour can even be found among \emph{simple}
   abelian surfaces: There is a two-dimensional family of
   principally polarized abelian surfaces $X$
   with endomorphism ring
   $\End(X)=\Z[\sqrt 2]$
   (cf.~\cite[Prop.~2.1]{Birkenhake:tensor}). On such a surface
   consider the endomorphism $f=-1+\sqrt 2$. Its analytic
   characteristic polynomial is $t^2+2t-1$, and hence it has
   $-1+\sqrt 2$ as an eigenvalue. Therefore $f^n$ has an eigenvalue
   arbitrarily close to 0 when $n$ is sufficiently large.
\end{example}

   In view of Prop.~\ref{prop:fix-and-eigenvalues},
   and considering the preceding examples, it becomes apparent that
   the issue, in the general case, is
   to understand what kind of eigenvalues
   may occur in endomorphisms of complex tori.
   We focus from now on on the surface case, where we show as
   a first step:

\begin{proposition}\label{prop:less-than-one}
   Let $X$ be a two-dimensional complex torus and $f:X\to X$ an
   endomorphism.
   If $f$ has a non-zero eigenvalue of absolute value $<1$, then it also
   has an eigenvalue of absolute value $>1$.
\end{proposition}

\begin{proof}
   The characteristic polynomial of the
   analytic representation $\rho_a(f)$ is
   \be
      P^a_f(t)=\det(t\unitmatrix_2-\rho_a(f))
         = (t-\lambda_1)(t-\lambda_2)
      \,.
   \ee
   Since the rational representation $\rho_r(f)$ is the direct
   sum of $\rho_a(f)$ and its
   conjugate,
   the
   characteristic polynomial of $\rho_r(f)$
   is given by
   \be
      P^r_f(t)=P^a_f(t)\cdot \bar P^a_f(t)
         =(t-\lambda_1)(t-\lambda_2)
          (t-\bar\lambda_1)(t-\bar\lambda_2)
   \ee
   Assume now by way of contradiction
   that
   $0<\lambda_1<1$ and $\lambda_2\le 1$.

   Consider first the case that $\lambda_2=0$. Then
   $\lambda_1\overline\lambda_1$ appears as a coefficient in
   $P^r_f(t)$, and it must therefore be an integer. So
   $|\lambda_1|\ge 1$, a contradiction.

   Next, suppose $\lambda_2>0$.
   We have
   \be
      \det(\rho_r(f))=\lambda_1\lambda_2\bar\lambda_1\bar\lambda_2
         = \abs{\lambda_1}^2\abs{\lambda_2}^2
      \,.
   \ee
   and this implies
   $0<\abs{\det(\rho_r(f))}<1$, which again is impossible as this
   number is an integer.
\end{proof}

   The next statement will be crucial for dealing with
   the case of eigenvalues of absolute
   value 1.

\begin{proposition}\label{prop:equal-one}
   Let $X$ be a two-dimensional complex torus and $f:X\to X$ an
   endomorphism.
   If $\lambda$ is an eigenvalue of $f$ with $\abs\lambda=1$,
   then $\lambda$ is a root of
   unity.
\end{proposition}

\begin{proof}
   We may assume that
   $\lambda\ne\pm 1$, as otherwise there is nothing to prove.
   Let $\lambda_1$ and $\lambda_2$ be the eigenvalues of $f$.
   Suppose to begin with that
   $\abs{\lambda_1}=\abs{\lambda_2}=1$.
   The characteristic polynomial
   $P^r_f$ of the rational representation of $f$ is then
   a monic polynomial over the integers,
   all of whose roots are of absolute value
   1. It follows
   that all of its
   roots are then roots of unity
   (see~\cite[Prop.~3.3.9]{Cohen:number-theory}),
   and we are done in this case.

   It remains to consider the case that $\abs{\lambda_1}=1$
   and $\abs{\lambda_2}\ne 1$.
   By Lemma~\ref{prop:less-than-one} we know that then
   necessarily
   $\abs{\lambda_2}> 1$ or $\lambda_2=0$.
   Let $h$ be the minimal polynomial of $\lambda_1$ over $\Q$.
   ($\lambda_1$ is a root of $P^r_f$, and hence an algebraic
   integer.)
   According to
   Lemma~\ref{lemma:mipo} below,
   $h$ is a symmetric integer polynomial of degree 2 or 4,
   whose roots appear in reciprocal pairs.
   So there are two cases:

   \textit{Case 1:} $\deg h=2$. In that case,
   the second root of $h$ is $\bar\lambda_1$.
   So all roots of $h$ are of absolute value 1.
   And as above, we conclude that $\lambda_1$ is a root of unity.

   \textit{Case 2:} $\deg h=4$. Then $h$ coincides with $P^r_f$,
   so its roots are $\lambda_1,\bar\lambda_1,\lambda_2,\bar\lambda_2$.
   But because of $\abs{\lambda_2}>\abs{\lambda_1}=1$, there is no
   way that the roots can occur in reciprocal pairs -- so this case
   does not happen.
\end{proof}

   We very much assume that the following elementary algebraic lemma
   is well-known.
   For lack of a reference we include a proof.

\begin{lemma}\label{lemma:mipo}
   Let $a\in\C$ be an algebraic integer
   of absolute value 1 and different from $\pm 1$.
   Then its minimal polynomial is a polynomial over $\Z$
   of even degree with symmetric coefficients,
   whose roots occur in reciprocal pairs.
\end{lemma}

\begin{proof}
   By definition there is a
   monic polynomial $g$ over $\Z$ with $g(a)=0$.
   As $g$ is a multiple of the minimal polynomial
   $h$ of $a$, it follows from
   Gau\ss' Lemma that $h$ is integral
   as well.

   We now prove the symmetry statement.
   As $a$ is of absolute value 1, we know that
   $1/a=\bar a$ appears as a root of $h$ as well.
   On the other hand, we have
   \be
      h(1/\bar a)=h(a)=0
      \,,
   \ee
   hence $\bar a$ is a root of the integral polynomial
   $t^n h(1/t)$, which is also of degree $n$.
   Therefore
   $t^n h(1/t)=c\cdot h(t)$
   for some $c\in\Q$.
   Setting $t=1$, we get $h(1)=c\cdot h(1)$.
   Since $a$ is irrational, the degree of $h$
   is at least 2, and hence
   $h(1)\ne 0$. We conclude that $c=1$, i.e.,
   \be
      t^n h(1/t)=h(t)
      \,,
   \ee
   and this shows that the roots of $h$ occur in
   reciprocal pairs, and hence that its degree is even.
\end{proof}

   We can now give the

\begin{proof}[Proof of Theorem~\ref{thm:tori}]
   Let $f:X\to X$ be an endomorphism of a two-dimensional complex
   torus, and let $\lambda_1$ and $\lambda_2$ be the eigenvalues of its
   analytic representation (counted with algebraic
   multiplicities). After reordering we assume $\abs{\lambda_1}\le\abs{\lambda_2}$.
   We have then by Prop.~\ref{prop:fix-and-eigenvalues}
   for every integer $n\ge 1$,
   \be
      \fix(f^n)=|(1-\lambda_1^n)(1-\lambda_2^n)|^2
      \,.
   \ee

   Suppose first that $\abs{\lambda_1}>1$. Then by our setup we
   have
   $\abs{\lambda_2}>1$ as well, and hence the fixed-points function
   $\fix(f^n)$ grows exponentially. So we are in Case~(B1)
   of the theorem.

   Suppose next that $0<\abs{\lambda_1}<1$. It follows from
   Prop.~\ref{prop:less-than-one} that then
   $\abs{\lambda_2}>1$, and hence the function $\fix(f^n)$ grows
   exponentially again.

   Suppose now that
   $\abs{\lambda_1}=1$.
   Proposition~\ref{prop:equal-one} tells us that $\lambda_1$ is then a
   root of unity. If $\abs{\lambda_2}=1$, then the same is true
   for $\lambda_2$ and hence $\fix(f^n)$ is a periodic function,
   so we are in Case~(B2) of the theorem.
   And if $\abs{\lambda_2}>1$, then the fixed-points function
   has the form described in Case~(B3) of the theorem.
   In either case, the roots of unity are of algebraic degree
   $\le 4$, since they appear as roots of the rational
   characteristic polynomial $P^r_f(f)$. They are therefore
   $k$-th roots of unity, where
   $k\in\set{1,\dots,6,8,10,12}$.

   Finally, suppose that $\lambda_1=0$. Then
   we have behaviour (B1) if $\abs{\lambda_2}>1$
   and behaviour (B2) if $\abs{\lambda_2}=1$.

   It remains to show that all three behaviours actually occur.
   Case~(B1) happens for the multiplication map $x\mapsto mx$
   on every complex
   torus, as soon as $\abs m\ge 2$.
   Example~\ref{example:eigenvalues-one} is an instance of
   (B2), and
   Example~\ref{example:eigenvalues-one-root} below
   shows that (B3) occurs.
   In all three cases there are projective examples.
\end{proof}

\begin{example} (One eigenvalue of absolute value $>1$, the other a root of unity.)
\label{example:eigenvalues-one-root}
   Consider the elliptic curve $E$ with complex multiplication
   in $\Z[i]$, and take the abelian surface $X=E\times E$.
   The endomorphism
   \be
      X\to X, \quad (x,y)\mapsto(ix, 2iy)
   \ee
   has the eigenvalues $i$ and $2i$, and hence the fixed-points function
   has the behaviour described in Case~(B3) of
   Theorem~\ref{thm:tori}:
   \be
      \fix(f^n)=
         \begin{bycases}
            0,   & \mbox{ if } n \equiv 0 \pmod 4 \\
            h(x) & \mbox{ otherwise }
         \end{bycases}
   \ee
   where the function $h$ grows exponentially, $h(n)\sim 2^{2n}$.
\end{example}

%*****************************************************************************

\section{Fixed points on simple abelian surfaces}\label{sect:abelian}

   In this section we will explicitly determine the endomorphisms
   on simple abelian surfaces whose fixed-points function grows
   exponentially, thus proving Theorem~\ref{thm:abelian}
   stated in the introduction.

   We will make use of the fixed-point formula for abelian varieties
   \cite[13.1.4]{BL:CAV}, which for
   an endomorphism $f$ of
   a simple abelian surface $X$
   takes the following form: Let $D=\End_\Q(X)$, $K=\mbox{center}(D)$,
   $e=[K:\Q]$, $d^2=[D:K]$, and let $N:D\to\Q$ be the reduced norm map.
   Then for $f\in\End(X)$,
   \be
      \fix(f)=\Big(N(1-f)\Big)^{\frac{4}{de}}
      \,.
   \ee
   (The formula expresses the fact that the characteristic
   polynomial of the rational representation $\rho_r(f)$
   coincides with the map $N^{{4}/{de}}$.)
   Since on a simple abelian surface every non-zero
   endomorphism $f$ is an isogeny, both eigenvalues of $f$
   are non-zero.

   We employ now a strategy as in \cite{BL:fixed}, i.e., we use
   Albert's classification and deal with the possible types
   of simple abelian surfaces separately.

%*****************************************************************************

\paragraph{Type 0: Integer multiplication.}
   Suppose that $\End(X)=\Z$. In that case every endomorphism $f$
   is a multiplication map $x\mapsto mx$ for some $m\in\Z$. So $f$
   has exponential fixed-points growth if and only if
   $\abs{m}>1$.

%*****************************************************************************

\paragraph{Type 1: Real multiplication.}
   Suppose that $X$ has real multiplication, i.e., that
   $\End_\Q(X)=\Q(\sqrt d)$ for some square-free integer $d>0$.
   Every endomorphism $f\in\End(X)$ is then of the form
   $f=a+b\omega$ with $a,b\in\Z$, where
   \be
      \omega=
      \begin{bycases}
         \sqrt d             & \mbox{ if } d\equiv 2,3 \tmod 4 \\
         \tfrac12(1+\sqrt d) & \mbox{ if } d\equiv 1 \tmod 4
      \end{bycases}
   \ee
   With respect to suitable coordinates on $\C^2$,
   the analytic representation $\rho_a:\End_\Q(X)=\Q(\sqrt d)\to M_2(\C)$ is given
   by
   \be
      1 \mapsto \unitmatrix_2 \quad\mbox{and}\quad
      \sqrt d \mapsto \matr{\sqrt d & 0 \\ 0 & -\sqrt d}
   \ee
   (see \cite{Ruppert:abelian}).
   So the eigenvalues of $\rho_a(f)=\rho_a(a+b\omega)$
   are
   \begin{itemize}\compact
   \item
      $a\pm b\sqrt d$, if $d\equiv 2,3 \tmod 4$,
   \item
      $a+b\tfrac12(1\pm\sqrt d)$, if $d\equiv 1 \tmod 4$
   \end{itemize}
   If $b=0$, then $f$ is multiplication by $a$, and hence it
   has exponential fixed-points growth if and only if $\abs a>1$.
   And if $b\ne 0$, then both eigenvalues are of absolute value
   $\ne 1$. Using Prop.~\ref{prop:less-than-one} we see that then
   $f$ has exponential fixed-points growth.

%*****************************************************************************

\paragraph{Type 2: Indefinite quaternion multiplication.}

   Suppose that $X$ has indefinite quaternion multiplication,
   i.e., there are $\alpha,\beta\in\Q\setminus\set 0$ with
   $\alpha\ge\beta$ and $\alpha>0$ such that
   $\End_\Q(X)$ is isomorphic to the quaternion algebra
   $\hilbertsymbol{\alpha}{\beta}{\Q}$, i.e.,
   $\End_\Q(X)=\Q+i\Q+j\Q+ij\Q$, where $i$ and $j$ satisfy the
   relations $i^2=\alpha$, $j^2=\beta$ and $ij=-ji$. Using the
   splitting field $\Q(\sqrt \alpha)$ of $\End_\Q(X)$, one has
   an isomorphism
   \be
      \psi:\End_\Q(X)\otimes_\Q\Q(\sqrt \alpha) \rightarrow M_2(\Q(\sqrt \alpha))
   \ee
   given by
   \be
      i\otimes1\mapsto\matr{\sqrt \alpha & 0 \\ 0 & -\sqrt \alpha} \quad
      \quad\mbox{and}\quad
      \quad j\otimes1\mapsto \matr{ 0 & \beta \\ 1 & 0}
      \,.
   \ee
   For an element $f\in\End(X)$, written as
   $f=a+bi+cj+dij$ with $a,b,c,d\in\Q$, we have
   \be
      \psi(f)=\matr{ a+b\sqrt \alpha & c\beta+d\beta\sqrt \alpha \\ c-d\sqrt \alpha & a-b\sqrt \alpha}
      \,.
   \ee
   The reduced norm of $f$ over $\Q$ is given by
   $N(f)=\det(\psi(f))=a^2-b^2\alpha-c^2\beta + d^2\alpha\beta$.
   After diagonalization of $\psi(f)$ the norm of
   $1-f^n$ can be written as
   \be
      \lefteqn{N(1-f^n)} & \\
                         & =  & \det\matr{1-(a+\sqrt{b^2\alpha+c^2\beta-d^2\alpha\beta})^n & 0 \\ 0 & 1-(a-\sqrt{b^2\alpha+c^2\beta-d^2\alpha\beta})^n} \\
                         & =  & (1-t_1^n)(1-t_2^n)
      \,,
   \ee
   where $t_i=a\pm\sqrt{b^2\alpha+c^2\beta-d^2\alpha\beta}$.
   So we have $\fix(f^n)=((1-t_1^n)(1-t_2^n))^2$.
   Consider now the \emph{reduced characteristic polynomial}
   of $f$,
   \be
      \chi_f(t)=t^2-\tr(f)t+N(f)=t^2-\tr(\psi(f))t+\det(\psi(f))
      \,.
   \ee
   Since $f$ is contained in an order, it is an integral element,
   and hence $\chi_f$ has integer coefficients.
   For every integer $m$ we have
   %\be
   %   \chi_f(m)^2 & = & (m-t_1)^2(m-t_2)^2=N(m-f)^2 \\
   %              & = & \fix(m-f)=\det(m\id_\Lambda-\rho_r(f)) \\
   %              & = & P^r_f(m)
   %   \,,
   %\ee
   \be
      \chi_f(m)^2 & = & (m-t_1)^2(m-t_2)^2=N(m-f)^2 \\
                  & = & N(1-(f-(m-1)))^2=\fix(f-(m-1)) \\
                  & = & \det(\unitmatrix_4-\rho_r(f-(m-1)))=\det(\unitmatrix_4-\rho_r(f)+(m-1)\unitmatrix_4) \\
                  & = & \det(m\unitmatrix_4-\rho_r(f))=P_f^r(m)
   \ee
   and this implies that $\chi_f(t)^2=P^r_f(t)$ as polynomials in $t$.
   Therefore, if we denote by $\lambda_1$ and $\lambda_2$ the analytic eigenvalues
   of $f$, then $t_1\in\set{\lambda_1,\overline{\lambda}_1}$
   and $t_2\in\set{\lambda_2,\overline{\lambda}_2}$.

   \begin{itemize}
    \item
      If $|t_1|>1$ and $|t_2|>1$, then
      $\fix(f^n)$ grows asymptotically with $|t_1t_2|^{2n}$.
   \item
      If $|t_1|<1$, then
      Prop.~\ref{prop:less-than-one}
      tells us
      that $\abs{\lambda_2}>1$.
      Therefore the number of
      fixed-points grows exponentially in this case.
   \item
      If $|t_1|=|t_2|=1$, then the possible real
      values for $t_1$ and $t_2$ are $\pm 1$. If $t_1$ and
      $t_2$ are complex, then they are roots of unity,
      because they are roots of the integral polynomial $\chi_f$.
      The function $n\mapsto \fix(f^n)$ is then periodic.
   \item
      We show now that the
      case $|t_1|>1$ and $|t_2|=1$ does not occur.
      Assume the contrary.
      Since $t_1$ and $t_2$ are the roots of a rational polynomial of
      degree $2$, we know that $t_2$ has to be $\pm 1$. As
      $t_1\neq t_2$, at least one of the coefficients $b,c,d$ is
      non-zero
      (since otherwise $a=t_1=t_2$).
      But then
      $t_2=\pm 1$ implies that
      $f\mp 1$ is a non-zero endomorphism whose norm is zero,
      and hence
      the
      quaternion algebra cannot be a division algebra,
      a contradiction.
   \end{itemize}

%*****************************************************************************

\paragraph{Type 3: Complex multiplication.}

   Suppose that $X$ has complex multiplication, i.e., that
   $\End_\Q(X)$ is isomorphic to
   an imaginary quadratic extension $K$ of a real
   quadratic number field $\Q(\sqrt d)$, where $d$ is a positive
   square-free integer.

   Consider an endomorphism
   $f\in\End(X)$
   and let $\lambda_1$ and $\lambda_2$ be the
   eigenvalues of its analytic representation $\rho_a(f)$.
   By the Cayley-Hamilton
   theorem, $f$ is
   annihilated by its rational characteristic polynomial
   $P^r_f$. As the isomorphism $\sigma:\End_\Q(X)\to K$ fixes $\Q$,
   the complex number $\fnum=\sigma(f)$ is then
   a zero of $P^r_f$
   as well. This implies that it is contained in the set
   $\set{\lambda_1,\lambda_2,\overline\lambda_1,\overline\lambda_2}$.
   So, after renumbering, we have $\fnum=\lambda_1$ or $\fnum=\overline\lambda_1$.
   We distinguish now cases according to $\fnum$:

   \begin{itemize}
   \item
      If $|\fnum|<1$, then we know by
      Prop.~\ref{prop:less-than-one} that $|\lambda_2|>1$.
      Therefore in this case the number of
      fixed points grows exponentially, which is behaviour (B1).

   \item
      If $|\fnum|=1$, then  by
      Prop.~\ref{prop:equal-one}
      we know that $\lambda_1$, and hence $\fnum$, is a root of unity.
      The fixed-points function
      $n\mapsto\fix(f^n)=N(1-f^n)$
      is then periodic (B2).

   \item
      Let $|\fnum|>1$ and $\fnum\in\R$. Then $\fnum$ is of the form
      $a+b\sqrt d$, hence we have behaviour (B1) as in the case
      of real multiplication.

   \item
      Let $|\fnum|>1$ and $\fnum\notin\R$.
      So $|\lambda_1|>1$.
      Our aim is to show that then $|\lambda_2|\ne 1$,
      which implies behaviour (B1).
      Assuming by way of contradiction that $|\lambda_2|=1$,
      note first that the proof of
      Prop.~\ref{prop:equal-one} shows that
      the minimal polynomial of $\lambda_1$ must be
      of degree 2. Therefore $\lambda_1$ can be written as
      $a+b\sqrt{-e}$ with $a,b\in\Q$ and $e$ a square-free positive
      integer. For $t\in\Z$ we compute the norm of $t-f$,
      \be
         N_{K/\Q}(t-f) & = & N_{\Q(\sqrt{-e})/\Q}(N_{K/\Q(\sqrt{-e})}(t-f)) \\
                       & = & N_{\Q(\sqrt{-e})/\Q}((t-f)^2) \\
                       & = & ((t-a)^2 + b^2e )^2
         \,.
      \ee
      On the other hand, $N_{K/\Q}(t-f)$ coincides as a polynomial in $t$
      with $P^r_f(t)$, and therefore
      %the minimal polynomial of $\lambda_1$ is $(t-a)^2 + b^2e$.
      $\lambda_2$ equals $\fnum$ or its conjugate.
      It cannot be of absolute value 1 then, and
      so we arrive at a contradiction.
   \end{itemize}

   We provide a complete list of the possible eigenvalues
   that occur for endomorphisms with periodic fixed-points
   behaviour in
   quaternion and complex multiplication.

\begin{proposition}\label{prop:periodic-cases}
\begin{itemize}
\item[\rm a)]
   If $X$ is a simple abelian surface with quaternion
   multiplication
   and $f$ a non-zero endomorphism on $X$ with periodic fixed-points
   behaviour, then all eigenvalues of $f$ are roots of unity of
   algebraic degree $\le 2$, i.e., they are contained in the following
   set:
   \be
   \begin{tabular}{ll}
      \tline
      in degree 1: & $\pm 1$ \\
      in degree 2: & $\textstyle \pm i, \quad \pm\frac{1}{2}\pm\frac{\sqrt{-3}}{2}$ \\
      \tline
   \end{tabular}
   \ee
   Conversely, each of these numbers
   occurs as an eigenvalue of an endomorphism
   on some simple abelian surface with quaternion
   multiplication.

\item[\rm b)]
   If $X$ is a simple abelian surface with complex
   multiplication
   and $f$ a non-zero endomorphism on $X$ with periodic fixed-points
   behaviour, then all eigenvalues of $f$ are roots of unity
   of algebraic degree $\le 4$, i.e., they are contained in the
   following set:
   \be
   \begin{tabular}{ll}
      \tline
      in degree 1: & $\pm 1$ \\
      in degree 2: & $\textstyle \pm i, \quad \pm\frac{1}{2}\pm\frac{\sqrt{-3}}{2}$ \\
      in degree 4: & $\textstyle \pm\frac{\sqrt 2}{2}\pm\frac{\sqrt{-2}}{2}, \quad \textstyle \pm\frac{\sqrt 3}{2}\pm\frac{\sqrt{-1}}{2},$ \\
                   & $\textstyle \pm(\frac{1}{4}+\frac{\sqrt 5}{4})\pm i\sqrt{\frac{5}{8}-\frac{\sqrt 5}{8}}, \quad \pm(\frac{1}{4}-\frac{\sqrt 5}{4})\pm i\sqrt{\frac{5}{8}+\frac{\sqrt 5}{8}}$ \\
      \tline
   \end{tabular}
   \ee
   Conversely, each of these numbers
   occurs as an eigenvalue of an endomorphism
   on some simple abelian surface with complex
   multiplication.
\end{itemize}
\end{proposition}

\begin{proof}
   One direction is clear by now:
   If $f$ has periodic fixed-points behaviour, then we know that
   all eigenvalues are roots of unity.
   In the quaternion case they are roots of $\chi_f$ and therefore
   of degree $\le 2$, and in the complex multiplication case they
   are roots of $P^r_f$ and therefore of degree $\le 4$.

   As for the converse statement in (a): The assertion being
   obvious for $\pm 1$, we now exhibit
   quaternion
   algebras $B_1$ and $B_2$ that are skew-fields,
   and
   orders $\mathcal
   O_1\subset B_1$ and
   $\mathcal O_2\subset B_2$
   containing
   elements that lead to the
   required
   eigenvalues $\pm i$ and $\pm\frac12\pm\frac{\sqrt{-3}}2$
   respectively.
   By Shimura's theory (cf.~\cite[\S9.4]{BL:CAV}), each of
   the orders $\mathcal O_k$ is contained in the endomorphism ring
   of some simple abelian surface.

   To this end, consider first
   the quaternion algebra
   $B_1=\hilbertsymbol{3}{2}{\Q}$.
   It
   is a skew field and contains the order
   $\mathcal O_1=\mathbb{Z}+\mathbb{Z}i+\mathbb{Z}j+\mathbb{Z}ij$.
   The element
   $f_1=i+j+ij$
   has reduced characteristic polynomial
   $\chi_{f_1}(t)=t^2+1$, whose roots are $\pm \sqrt{-1}$.

   Secondly, consider
   $B_2=\hilbertsymbol{-3}{2}{\Q}$ and
   the splitting field $L=\Q(i)$, where $i^2=-3$.
   The maximal order in $L$ is $S=\Z+\frac{1+i}2\Z$, and it
   follows from this that $S+jS$, with $j^2=2$, is an order in $B_2$.
   The element $f_2=\frac{1+i}{2}\in S$ hat reduced characteristic
   polynomial
   $\chi_{f_2}(t)=t^2-t+1$, whose roots are
   $\frac12\pm\frac12\sqrt{-3}$.
   And the element $f_3=\frac{-1+i}{2}\in S$ has reduced characteristic
   polynomial
   $\chi_{f_3}(t)=t^2+t+1$, whose roots are
   $-\frac12\pm\frac12\sqrt{-3}$.

   As for the converse statement in (b): Every root of unity
   $\zeta$ of
   algebraic degree 4 is an algebraic integer in the CM field
   $\Q(\zeta)$, and therefore
   there exists a simple abelian surface $X$
   of CM-type, where the root represents an endomorphism $f$
   (cf.~\cite[\S9.6]{BL:CAV}).
   For the roots of unity of degree 2
   just consider any real
   quadratic number field $L$ and take the CM field
   $L(\zeta)$.
\end{proof}

% <<end of main text>>

%*****************************************************************************
% Bibliography

%***************************************************************************** % Addresses

\footnotesize
   \bigskip
   Thomas Bauer,
   Fachbereich Mathematik und Informatik,
   Philipps-Universit\"at Marburg,
   Hans-Meerwein-Stra\ss e,
   D-35032 Marburg, Germany.

   \nopagebreak
   \textit{E-mail address:} \texttt{tbauer@mathematik.uni-marburg.de}

   \bigskip
   Thorsten Herrig,
   Fachbereich Mathematik und Informatik,
   Philipps-Universit\"at Marburg,
   Hans-Meerwein-Stra\ss e,
   D-35032 Marburg, Germany.

   \nopagebreak
   \textit{E-mail address:} \texttt{herrig106@mathematik.uni-marburg.de}

%*****************************************************************************

\end{document}